\documentclass[11pt]{amsart}
\usepackage{amsmath,amsfonts,amssymb}
\newtheorem{theorem}{Theorem}[section]
\newtheorem{lemma}[theorem]{Lemma}
\newtheorem{proposition}[theorem]{Proposition}
\newtheorem{remark}[theorem]{Remark}
\numberwithin{equation}{section}
\newtheorem{definition}[theorem]{Definition}

\newtheorem{corollary}[theorem]{Corollary}

\input xy
\xyoption{all}

\begin{document}

\title {Irreducibility of moduli of vector bundles over a very general sextic Surface}


\author[S.Pal]{Sarbeswar Pal}

\address{Indian Institute of Science Education and Research, Thiruvananthapuram,
Maruthamala PO, Vithura,
Thiruvananthapuram - 695551, Kerala, India.}

\email{sarbeswar11@gmail.com,  spal@iisertvm.ac.in}

\date{}

\begin{abstract}
Let $S$ be a very general smooth hypersurface of degree $6$ in $\mathbb{P}^3$. In this paper we will prove that the moduli
space of $\mu$-stable rank $2$ torsion free sheaves with respect to hyperplane section having $c_1 = \mathcal{O}_S(1)$,
 with fixed, $c_2 \ge 27$ is irreducible. 
\end{abstract}

\maketitle

\section{Introduction}
Let $S$ be a projective irreducible smooth surface over $\mathbb C$  and $H$ an ample divisor on S.
Let $r \ge 1$ be an integer, $L$ a line bundle on $S$, and $c_2 \in H^4(S,\mathbb Z)\,\simeq \,\mathbb Z$.  The moduli space of semistable torsion free sheaves (w.r.t $H$) with fixed determinant $L$
and second Chern class $c_2$ was first constructed by Gieseker and Maruyama (see \cite{GM}, \cite{MAR}) using Mumford's geometric invariant theory 
and it's a projective scheme 
(need not be reduced). After the construction, many people have studied the geometry of this moduli space. The study has been done by fixing the underlying surface. For example when the surface is rational Barth \cite{WB}, Rosa Maria \cite{RM},  Le Potier \cite{LP} 
proved that the moduli space is reduced, irreducible and rational under certain condition on rank and Chern classes. When the surface is K3, it has been studied by Mukai \cite{SM} and many others.  When the surface is general, Jun Li \cite{JL} showed that for $c_2$ big enough, the moduli space is also of general type. The guiding general philosophy  
is that the geometry of the moduli space is reflected by the underlying geometry of the surface. The first result without fixing the underlying surface was given by O'Grady. In \cite{OG} O'Grady proved that for sufficiently large second Chern class $c_2$, the moduli space is reduced, generically smooth and irreducible. In fact, O'Grady's first step to prove irreducibility is to show each component is generically smooth of expected dimension. The generic smoothness result was also proved by Donaldson \cite{D} in the rank $2$ case and trivial determinant and Zuo \cite{ZUO} for arbitrary determinant.

After O'Grady's result, it was important to give an effective bound on $c_2$ for the irreducibility and generic smoothness of the moduli space.  The moduli space of vector bundles over hypersurfaces is one of the important objects to study.
When the underlying surface is a very general quintic hypersurface in $\mathbb P^3$ Simpson and Mestrano studied this question systematically and in \cite{KS}, the current author, with K. Dan, studied the question related to Brill-Noether loci.   In a series of paper \cite{SIM}, \cite{SIM2}, \cite{SIM3}, Simpson and Mestrano proved that moduli space of rank $2$ torsion free sheaves is generically smooth, irreducible. This result was known before by an unpublished work by Nijsse for $c_2 \ge 16$ \cite{N}.
  
Since rank and degree of the underlying sheaf are co-prime the stability is equivalent to semistability.
Motivated by the result of Mestrano and Simpson  we look at the next case i.e. Moduli space of rank $2$ torsion free sheaves on a very general sextic surface $S$, that is a very general hypersurface  of degree $6$  in 
$\mathbb{P}^3$. 

In \cite{SIM}, Simpson and Mestrano showed that the moduli space of stable rank $2$ bundles over a very general sextic surface 
is not irreducible for $c_2 =11$. The natural question arises, can one give an effective bound for $c_2$ such that the moduli space
become irreducible. In this paper, we answer this question.

Let $S \subset \mathbb{P}^3$ be a very general sextic surface over $\mathbb C$ and $H$ denote the very ample line bundle $\mathcal O_S(1)$. The Picard group of $S$ is generated by $H$. Let  $\mathcal{M}(H,c_2)$ denote the moduli space of
$H$-stable rank $2$ locally free sheaves on $S$ with fixed determinant isomorphic to $H$ and the second Chern class is $c_2$ and $\overline{\mathcal{M}(H, c_2)}$ be its closure in the Gieseker-Maruyama moduli space $\overline{\mathcal{M}}(H, c_2)$  of semistable torsion free sheaves. 

It is known that $\overline{\mathcal{M}}(H, c_2)$ is projective and $\mathcal{M}(H, c_2)$ sits inside
$\overline{\mathcal{M}}(H, c_2)$
as an open subset, whose complement is called the \textit{boundary}.

In this paper, we will give a lower bound for $c_2$ such that the moduli space $\overline{\mathcal{M}}(H, c_2)$ is 
irreducible.
\begin{theorem}
 For $c_2 \ge  27$, the moduli space $\overline{\mathcal{M}}(H, c_2)$ is irreducible.
\end{theorem}
\subsection{Outline of the proof and the organization of the paper:}

Our technique is to use O'Grady's method of deformation to the boundary \cite{OG}, \cite{OG2}, as it was exploited by Nijsse \cite{N} in the case of a very general quintic hypersurface. In particular, we prove a connectedness Theorem of $\overline{\mathcal{M}}(H, c_2)$  as in Nijsse \cite{N}. To use Nijsse's argument in connectedness Theorem, which we prove in section \ref{S4}, we  need to prove certain HN-strata is non-empty, which we consider in  section \ref{S2}.  Using non-emptyness of the HN-strata, in section \ref{S3} we also prove the non-emptyness of the boundary for $c_2 \ge 34$. 

In section \ref{S1}, we will show that $\overline{\mathcal{M}}(H, c_2)$ is good for $c_2 \ge 27$. Then Kuranishi theory of deformation
spaces implies 
that $\overline{\mathcal{M}}(H, c_2)$ is locally a complete intersection. Hartshorne's connectedness theorem \cite{HA} 
now says that if two
different irreducible components of $\overline{\mathcal{M}}(H, c_2)$ meet at some point, then they intersect in a 
codimension 
1 subvariety. This intersection has to be contained in the singular locus.
We also show that the  singular locus of $\mathcal{M}(H, c_2)$  is the union of the locus $V$ consisting of 
bundles $E$,
with $h^0(E) >0$, which has dimension $3c_2 -21$, plus other pieces of dimension $\le 39$. The dimension bound of $V$ together
with the connectedness of $\overline{\mathcal{M}}(H, c_2)$ allow us to show that any two distinct irreducible components 
intersects
in boundary in codimension 1.
Then we show that the boundary has a stratification as the union of $\mathcal{M}(c_2, c_2-1)$ which is of codimension 1 
and other strata with codimension strictly greater than 1, where $\mathcal{M}(c_2, c_2-1)$ is defined in section 5. 
This fact forces that  two distinct components of $\overline{\mathcal{M}}(H, c_2)$ intersects $\mathcal{M}(c_2, c_2-1)$ 
in an open
set and contained in the singular locus. Then the irreducibility follows from the fact that a general point$\mathcal{M}(c_2, c_2-1)$ 
is a smooth point of $\overline{\mathcal{M}}(H, c_2)$.

\section{Generic smoothness of the moduli space}\label{S1}

Let $S$ be a smooth irreducible projective surface over the field of complex numbers.
Let $\mathcal{M}(r, H, c_2)$ be the moduli space of rank $r$, $H$-stable vector bundles with fixed determinant $H$ and second
Chern class $c_2$, where $H:= c_1(\mathcal{O}_S(1))$. Consider a point $E \in \mathcal{M}(r, H, c_2)$.
Then the obstruction theory is controlled by 
\[
 \text{Obs}(E) := H^2(S, \text{End}^0(E)).
\]
Here $\text{End}^0(E):= \text{Ker}(\text{tr}: \text{End}(E) \longrightarrow \mathcal{O}_S).$\\
By Serre duality,
\[
 H^2(S, \text{End}^0(E)) \cong H^0(S, \text{End}^0(E) \otimes K_S)^*.
\]
So, $\text{Obs}(E) \ne 0$ if and only if, there exists a non-zero element $\varphi \in H^0(S, \text{End}^0(E) \otimes K_S)$.
In other words, there exists a twisted endomorphism 
\[
 \varphi: E \longrightarrow E \otimes K_S, \text{with } \text{tr}(\varphi) = 0.
\]

Here, we only consider the case when $r=2$.

Let $\Sigma(c_2):= \{ E \in \mathcal{M}(r, H, c_2):  H^0(S, \text{End}^0(E) \otimes K_S) \ne 0\}$. We break $\Sigma(c_2)$ into
two parts.
\[
 \Sigma_1:= \{ E \in \Sigma(c_2): \exists s(\ne0) \in  H^0(S, \text{End}^0(E) \otimes K_S) \text{ and }\exists
 t \in H^0(S, K_S) \text{such that} \text{ det}(s) - t^2 = 0\}.
\]
and
\[
 \Sigma_2:= \Sigma(c_2) \setminus \Sigma_1.
\]

Case(1): Suppose $E \in \Sigma_1$. Then there exists a non-zero section $s \in  H^0(S, \text{End}^0(E) \otimes K_S)$ and 
$t \in H^0(S,  K_S)$ such that $\text{det}(s) - t^2 = 0$. Now consider the non-zero maps 
$\alpha_1:= s + \text{Id} \otimes t: E \longrightarrow E \otimes K_S$ and
$\alpha_2:=s - \text{Id} \otimes t: E \longrightarrow E \otimes K_S$. 

Note that $\alpha_2 \circ \alpha_1 = s^2 -t^2$. 
On the other hand, by Cayley-Hamilton theorem, $s$ satisfies the equation, $x^2 - \text{Tracce}(s) x +\text{det}(s) $. 
Since $s$ is trace free endomorphism, and $\text{det}(s) = t^2$, 
we have $s^2 -t^2 = 0$. 
Thus $\alpha_2 \circ \alpha_1 = 0$. 
Since, $\text{det}(\alpha_i) = 0, i = 1, 2$,
 we have the following diagram,
 
 \[
  \xymatrix{
& 0 \ar[r] & L_1 \ar[r] & E \ar[r] \ar[d]^{\alpha_1} & L_1^* \otimes \mathcal{O}_S(1) \otimes \mathcal{J}_{P_1} \ar[r]  & 0\\
& 0 \ar[r] & L_2 \otimes K_S  \ar[r]  & E \otimes K_S \ar[r] \ar[d]^{\alpha_2 \otimes Id} & L_2^* \otimes \mathcal{O}_S(1) \otimes K_S \otimes \mathcal{J}_{P_2} \ar[r] & 0, \\
& & & E \otimes K_S^2 & &\\
} 
 \]
 
where $L_i = \text{ker}(\alpha_i)$ and  $\mathcal{J}_{P_i}:= L_i \otimes \mathcal{O}_S(-1) \otimes (E/L_i)$. It is the ideal in $\mathcal{O}_S$ of a zero dimensional 
subscheme $P_i$.
Since the composition map $\alpha_2 \circ \alpha_1$ is zero, there is a  factorization through 
\[
 L_1^* \otimes \mathcal{O}_S(1) \otimes \mathcal{J}_{P_1} \longrightarrow L_2 \otimes K_S.
\]

Case(2):  If $E \in \Sigma_2$, then there exists a non-zero section $s \in H^0(S, \text{End}^0(E) \otimes K_S)$ with non-zero
determinant and which is not a square. In this case, as in \cite[Section 2]{ZUO}, $\text{det}(s)$ defines a double cover
$ r: Z \longrightarrow S$ with $Z \subset K_S$ and $r$ is ramified along  a non-zero subdivisor of  zeros of $\text{det}(s)$, together with a line 
bundle $L$ over a desingularization $\varepsilon: \tilde{Z} \longrightarrow Z$ such that $E = r_*\varepsilon_*(L)^{**}$.

We now apply this for sextic surface.
Let $S \subset \mathbb{P}^3$ be a  very general smooth surface of degree $6$. Then the canonical line bundle $K_S \simeq \mathcal{O}_S(2)$.
We denote by $\mathcal{M}(H, c_2)$ the moduli space of rank 2, $  \mu- $ stable vector bundles $E$ on $S$,
with respect to $H := \mathcal{O}_S(1)$ and with $c_1(E)=H$ and $c_2(E) =c_2$.
 Note that in this case, the stability and semistability are the same. The expected dimension of $\mathcal{M}(H, c_2)$ is
 $4c_2 -c_1^2 - 3\chi(\mathcal{O}_S) = 4c_2 -39$.
In this section, we will show that 
 the moduli space $\mathcal{M}(H, c_2)$ is generically smooth for $c_2 \ge 20$.
\begin{proposition}\label{Prop1}
 Suppose $S$ is as above. Suppose $E$ is a stable bundle of rank $2$ with 
 $\text{det}(E) \cong \mathcal{O}_S(1)$ and $E$ is potentially obstructed, that is, $H^0(S, \text{End}^0(E) \otimes K_S) \ne 0$. Then
 either\\
 (I) $H^0(S, E) \ne 0$

 or,
 
 (II) There is a section $\beta \in H^0(S, K_S^2)$ which is not a square defining 
 a double cover
$ r: Z \longrightarrow S$ with $Z \subset K_S$ and $r$ is ramified along $\text{zero}(\text{det}(s))$, together with a line 
bundle $L$ over a desingularization $\varepsilon: \tilde{Z} \longrightarrow Z$ such that $E = r_*\varepsilon_*(L)^{**}$.

\end{proposition}
\begin{proof}
 Since $H^0(S, \text{End}^0(E) \otimes K_S) \ne 0$, there is a twisted morphism $\varphi: E \longrightarrow E \otimes K_S$.
 Let $\beta := \text{det}(\varphi)$. If $\beta$ is not a square then by the discussion earlier in Case(2), we are in case (II).
 
 If $\beta = t^2$, for some $t \in H^0(S, K_S)$, then by the discussion in Case(1), $E$ fits into exact sequences of the form,
 \[
  0 \longrightarrow L_1 \longrightarrow E \longrightarrow L_1^* \otimes \mathcal{O}_S(1) \otimes \mathcal{J}_{P_1} \longrightarrow 0
 \]
 and
 \[
  0 \longrightarrow L_2 \longrightarrow E \longrightarrow L_2^* \otimes \mathcal{O}_S(1) \otimes \mathcal{J}_{P_2} \longrightarrow 0.
 \]

 Therefore, we get a non-zero homomorphism from $L_1^* \otimes \mathcal{O}_S(1) \otimes \mathcal{J}_{P_1} \longrightarrow L_2 \otimes K_S$.
 Since $S$ is very general, $\text{Pic}(S)= \mathbb{Z}$ and hence  $L_1$ is of the form, $\mathcal{O}_S(k)$ for some integer $k$
 and $L_2$ is of the form $\mathcal{O}_S(m)$. But since 
 $E$ is stable, $k, m \le 0$. Now, the existence of a non-zero homomorphism from
 $ L_1^* \otimes \mathcal{O}_S(1) \otimes \mathcal{J}_{P_1} \longrightarrow L_2 \otimes K_S$ tells us that 
 $k, m \in \{0, -1\}$. If both $k \text{ and } m$ are $-1$, then the non-zero homomorphism, $L_1^* \otimes \mathcal{O}_S(1) \otimes \mathcal{J}_P \longrightarrow L_2 \otimes K_S$, will give a non-zero homomorphism from $\mathcal{J}_{P_1} \to \mathcal{O}_S(-1)$, which is a contradiction. 
 Thus
 $H^0(S, E) \ne 0.$

\end{proof}

\subsection{Dimension of $\Sigma(c_2)$}
Define $V$ as $V := \{[F] \in {\mathcal{M}}(H, c_2): H^0(S, F) \ne 0\}$.
\begin{proposition}\label{p1}
For $c_2 \ge 20$, $ \text{ dim }V \le 3c_2 - 21$.
 \end{proposition}

 \begin{proof}
Let $k$ be the smallest positive integer such that there is an element $E \in V \text{ with } h^0(E) = k$. Let $V_k:= \{ E \in V \text{ such that } h^0(E) = k\}$.
Let $\mathcal{N}(V_k)$ be the space of pairs,
\[
 \mathcal{N}(V_k)= \{(E, s): E \in V_k, s \in \mathbb{P}(H^0(S, E))\}.
\]
Note that for each $E \in V_k, h^0(E)$ is constant. Assuming $\mathcal{M}(H, c_2)$ is fine moduli space and $\mathcal{U}$ is the universal bundle on $\mathcal{M}(H, c_2) \times S$ one can view the space of pairs $\mathcal{N}(V_k)$ as the total space of the bundle $\mathbb{P}(p_*(\mathcal{U} \vert_{V_k}))$, where $p$ is the projection map on to $\mathcal{M}(H, c_2)$. When the moduli space is not fine, it is possible to carry out this construction locally and then can be glued together to get a global algebraic object. Let $\mathcal{N}(V)= \cup_k \mathcal{N}(V_k)$. We claim that this union is a finite union.\\
If $E \in V$, then $E$ fits into an exact sequence 
  \[
   0 \longrightarrow \mathcal{O} \longrightarrow E \longrightarrow \mathcal{J}_Z(1) \longrightarrow 0,
  \]
where $Z \in Hilb^{c_2}(S)$ satisfying Cayley-Bacharach property for sections of $\mathcal{O}_S(3)$ and $\mathcal{J}_Z$ is the ideal sheaf corresponding to $Z$. Thus $h^0(E) \le 1+h^0(\mathcal{J}_Z(1) \le 5$, which proves the claim.\\

Now consider the following diagram
 \[
\xymatrix{
& \mathcal{N}(V) \ar[r]^{p_2} \ar[d]^{p_1} & \text{Hilb}^{c_2}(S)\\
& V\\
}  
\]
Clearly $p_1$ is surjective. Thus $\text{dim }V \le \text{dim}(\mathcal{N}(V))$. On the other hand, 
\[
 \text{dim }p_2^{-1}(Z) = \text{dim }\mathbb{P}(\text{Ext}^1(\mathcal{J}_Z(1), \mathcal{O})) .
\]
Now $\text{Ext}^1(\mathcal{J}_Z(1), \mathcal{O}) = \text{Ext}^1(\mathcal{J}_Z(3), \mathcal{O}(2))$, and 
by Serre duality, we have, $ \text{Ext}^1(\mathcal{J}_Z(3), \mathcal{O}(2)) = h^1(\mathcal{J}_Z(3))$. Thus $\text{dim}(p_2^{-1}(Z)) = h^1(\mathcal{J}_Z(3))-1$.\\
For $c_2 \ge 20$, a general $Z$  of length $c_2$ has $h^0(\mathcal{J}_Z(3)) = 0$, so 
using the fact that $h^0(\mathcal{O}(3)) = 20$ and from the canonical exact sequence 
\[
 0 \longrightarrow \mathcal{J}_Z(3) \longrightarrow \mathcal{O}(3) \longrightarrow \mathcal{O}_Z(3) \longrightarrow 0,
\]
we have $h^1(\mathcal{J}_Z(3)) = c_2 - 20$. 
Now $\text{Hilb}^{c_2}(S)$ has dimension $2c_2$. Thus over a dense open subset $U \subset \text{Hilb}^{c_2}(S), \text{ dim }( p_2^{-1}(U)) \le 3c_2 -21$.  Thus to conclude the proposition, it is enough to show that over a proper closed subset $W \subset \text{Hilb}^{c_2}(S) $, the dimension of $p_2^{-1}(W)$ does not jump.

To see this, let us consider other subsets, $\triangle_i =\{ Z \in \text{Hilb}^{c_2}(S): h^0(\mathcal{J}_Z(3)) \ge i \}$.
Consider the incidence variety $T = \{(C, Z): Z \subset C\} \subset \mathbb{P}(H^0(\mathcal{O}_S(3))) \times \text{Hilb}^{c_2}(S)$
and let $\pi_1, \pi_2$ are projections. The dimension
of
$\pi_1^{-1}(C)$ is at most $c_2$.\\
The proof for this part is taken from \cite{OG1}.
 Let $\tilde{C} \in \mathbb{P}(H^0(\mathcal{O}_S(3))$. For a point $x\in \tilde{C}$ and any $m\in \mathbb{N}$ denote by $\text{Hilb}^m(\tilde{C}, x) \subset \text{Hilb}^m(\tilde{C})$ the locus consisting of subschemes supported only at $x$. Since $S \subset \mathbb{P}^3$ is smooth, the singularities of $\tilde{C}$ are planar, so that for every $c\in \tilde{C}$ the locus $\text{Hilb}^m(\tilde{C}, x)$ can be identified with a subscheme of $\text{Hilb}^m(\mathbb{A}^2, 0)$.
On the other hand, Briancon \cite{BA}  proved that in characteristic zero, $\dim \text{Hilb}^m(\mathbb{A}^2, 0) = m-1$, so $\text{Hilb}^m(\tilde{C}, x)\leq m-1$. 

Now fix a partition $\lambda = (\lambda_1,\ldots, \lambda_k)$ of $c_2$ and consider the locus $S_{\lambda}$ in $\text{Sym}^{c_2}(\tilde{C})$ consisting of cycles of the form $\sum \lambda_i x_i$ for distinct $x_i\in \tilde{C}$. By the above estimate, the fiber of Hilbert-Chow morphism over any element of $S_{\lambda}$ has dimension at most $c_2-k$. The locus $S_{\lambda}$ has codimension $c_2-k$ in $\text{Sym}^{c_2}(\tilde{C})$, so we get that its preimage has dimension at most $c_2$, summing over all $\lambda$ we get the claim.

 Therefore, we have,  $19 + c_2 \ge \text{dim }T \ge \text{dim }\pi_2^{-1}(\triangle_i) \ge \text{dim }\triangle_i +i-1$.
This implies that $\text{dim }\triangle_i$ is bounded by $c_2 +20 -i \le 2c_2-i$, as $c_2 \ge 20$ and hence the codimension of $\triangle_i$ in
$\text{Hilb}^{c_2}(S)$ is $\ge i$.
 Thus $\text{dim }\mathcal{N}(V) \le 3c_2 -21$.
\end{proof}

Let $A$ be the maximum of $\text{dim }\text{Pic}^0(\tilde{Z})$, where $\tilde{Z}$ arising as the desingularization of a 
spectral cover $Z$ associated to a non-zero section $\beta \in H^0(S, K_S^2)$. The dimension of the space of $\beta$ is 
bounded by $34$. Thus, the dimension of space of potentially obstructed bundles in Case (II) is bounded by $A + 34$. However, 
at a general choice of $\beta \in H^0(S, K_S^2)$ the spectral cover $Z$ is itself  smooth and has irregularity $0$. Thus, 
the dimension bound can be reduced to $A + 33$.

We now give an estimate for the irregularity $A$.
\begin{lemma}
$A \le 6$.
\begin{proof}
We basically mimic the calculation done in  \cite[Lemma 9.1]{SIM} with minor changes.

  Let $X$ be a very general sextic hypersurface in $\mathbb{P}^3$ and $s \in H^0(X, \mathcal{O}(4))$ be a section which is not the 
 square of a section of $\mathcal{O}_X(2)$. The number $A \le h^0( \tilde{Z}, \Omega_{\tilde{Z}}^1)$. \\
 As in \cite{SIM} Lemma 9.1, we have $p_*(\Omega_{\tilde{Z}}^1) = \Omega_X^1 \oplus \mathcal{F}$, where 
 $p: \tilde{Z} \longrightarrow X$ is the projection map and $\mathcal{F}$ is a torsion free sheaf. If $\mathcal{G}$ denotes the double
 dual of $\mathcal{F}$, then as in \cite{SIM} Lemma 9.1 we also have an exact sequence
 \[
  0 \longrightarrow \Omega_X^1 \otimes L \longrightarrow \mathcal{G} \longrightarrow \mathcal{B} \longrightarrow 0,
 \]
where $\mathcal{B}$ is a sheaf supported on the divisor $D$ of zeros of $s$ and $L$ is a line bundle isomorphic to $\mathcal{O}_X(-2)$.
From the above exact sequence we get 
\[
 0 \longrightarrow \mathcal{G} \longrightarrow \Omega_X^1(-2)(D) =\Omega_X^1(2).
\]
Since $H^0(X, \Omega_X^1) = 0$, we have 
\[
 H^0(\tilde{Z}, \Omega_{\tilde{Z}}^1) \cong H^0(X, p_*\Omega_{\tilde{Z}}^1) \cong H^0(X, \mathcal{F}) \hookrightarrow H^0(X, \mathcal{G}) \hookrightarrow H^0(X, \Omega_X^1(2)).
\]
From the canonical exact sequence,
\[
 0 \longrightarrow \Omega_{\mathbb{P}^3}^1 \longrightarrow \mathcal{O}(-1)^4 \longrightarrow \mathcal{O}_{\mathbb{P}^3} \longrightarrow 0
\]
we observe that $H^0(\Omega_{\mathbb{P}^3}^1(1)) = H^1(\Omega_{\mathbb{P}^3}^1(1)) = H^1(\Omega_{\mathbb{P}^3}^1(-4)) = 0$.
Thus, from the exact sequences
\[
 0 \longrightarrow \Omega_{\mathbb{P}^3}^1(-4) \longrightarrow \Omega_{\mathbb{P}^3}^1(2) \longrightarrow \Omega_{\mathbb{P}^3}^1(2)_{\mid{X}} \longrightarrow 0,
 \]
 \[
  0 \longrightarrow N_{X/\mathbb{P}^3}^*(2) = \mathcal{O}_X(-4) \longrightarrow \Omega_{\mathbb{P}^3}^1(2)_{\mid{X}} \longrightarrow \Omega_X^1(2) \longrightarrow 0
 \]
and using the fact that $H^1(\mathcal{O}_X(n)) =0$ we have,
$H^0(\Omega_X^1(2)) \cong H^0(\Omega_{\mathbb{P}^3}^1(2))$.
On the other hand, from the exact sequences,
\[
 0 \longrightarrow \Omega_{\mathbb{P}^3}^1(1) \longrightarrow \Omega_{\mathbb{P}^3}^1(2) \longrightarrow \Omega_{\mathbb{P}^3}^1(2)_{\mid{\mathbb{P}^2}} \longrightarrow 0
\]
and
\[
 0 \longrightarrow N_{\mathbb{P}^2/\mathbb{P}^3}^*(2)= \mathcal{O}_{\mathbb{P}^2}(1) \longrightarrow \Omega_{\mathbb{P}^3}^1(2)_{\mid{\mathbb{P}^2}} \longrightarrow \Omega_{\mathbb{P}^2}(2) \longrightarrow 0
\]
we have, $h^0(\Omega_{\mathbb{P}^3}^1(2)) = 3 + h^0(\Omega_{\mathbb{P}^2}^1(2))$.
Using similar exact sequences for the embedding of $\mathbb{P}^1$ in $\mathbb{P}^2$ one can easily show that $h^0(\Omega_{\mathbb{P}^2}^1(2)) = 3$.\\

Therefore, $h^0(\Omega_{\mathbb{P}^3}^1(2)) =6$ from which the Lemma follows.

\end{proof}

\end{lemma}

Thus, we have the following Proposition:
\begin{proposition}
 The dimension of the locus of potentially obstructed bundles of type (II) is $\le 39$.
\end{proposition}

On the other hand, by the Proposition \ref{Prop1}, $\Sigma_1$ is contained in $V$ and by Proposition \ref{p1}, 
$\text{dim}(V) \le 3c_2 -21$  for $c_2 \ge 20$ and if $c_2 \ge 20$ then $3c_2 -21 \ge 39$.
Thus, the dimension of $\Sigma(c_2)$ is less than equal to $3c_2 -21$ for $c_2 \ge 20$. Again, by \cite[Corollary 3.1]{SIM2}, 
$\Sigma_1$ has dimension at least $3c_2 -21$ and hence it has dimension equal to $3c_2 -21$. Thus, we have the following proposition: 
\begin{proposition}\label{p2}
 The dimension of $\Sigma(c_2)$ is equal to $3c_2 -21$ for $c_2 \ge 20$.
\end{proposition}

\begin{definition}
 A closed subset $X \subset {\mathcal{M}}(H, c_2)$ is called good if every irreducible component of $X$ contains a point
 $[E]$ with $H^2(S, \text{End}^0 E) = 0$, where $\text{End}^0E$ denotes the traceless endomorphisms of $E$.
\end{definition}
 Now we have the following theorem:
\begin{theorem}\label{thm1}
 The moduli space $\mathcal{M}(H, c_2)$ is good for $c_2 \ge 20$ 
\end{theorem}
\begin{proof}
 The expected dimension of the moduli space is $4c_2-39$. Therefore, by Proposition \ref{p2}  the locus of potentially obstructed bundles has codimension
 at least $1$ for $c_2 \ge 20$. Thus, each component of  the moduli space contains a point $[E]$ with $H^2(S, \text{End}^0E) = 0$.  Hence,  $\mathcal{M}(H, c_2)$ is good. 
\end{proof}
\begin{remark}\label{RZ}
Note that, since $\mathcal{M}(H, c_2)$ is good for $c_2 \ge 20$, the moduli space $\overline{\mathcal{M}}(H, c_2)$ is also good for $c_2 \ge 27$,  \cite[Theorem 7.2]{DS}.
\end{remark}

\section{Non emptyness of certain HN-strata}\label{S2}
In this section, using O'Grady's method used in \cite{OG}, we prove the non-emptyness of certain HN-strata of a certain closed subset $X_C $ of an irreducible component $X \subset \mathcal{M}(H, c_2)$.
Let $C \in |H|$ be a smooth curve. Then genus $g_C= \frac{1}{2}(C.K_S+C.C) +1=10$. Let $X$ be an irreducible component of $\mathcal{M}(H, c_2)$. Let $\overline{X}$ be the closure of $X$ in $\overline{\mathcal{M}}(H, c_2)$.
Assume that the restriction of the torsion free sheaves in $\overline{X}$ to $C$ are locally free.
\begin{proposition}\label{P1}
   If $c_2 \ge 17$, then there is a torsion free sheaf $[E] \in \overline{X}$ such that
  $E_{ \mid_C}$ is not semistable.
  \end{proposition}
  
  \begin{proof}
  Let $\mathcal{M}(C, H_{\mid_C})$ be the moduli space of rank $2$ semistable bundles on $C$ with determinant $H_{\mid_C}$. If 
  for every $E$, $E_{\mid_C}$ is semistable then we have a restriction map
  \[
   \rho: \overline{X} \longrightarrow \mathcal{M}(C, H_{\mid_C}).
  \]
Now the dimension of $\mathcal{M}(C, H_{\mid_C})$ is $3g_C -3= 27$. But for $c_2 \ge 17$ the dimension of $X \ge 29$. Thus,
\[
 (\rho^*\Theta)^{\text{dim}\overline{X}} =0,
\]
where $\Theta$ denotes the theta divisor on $\mathcal{M}(C, H_{\mid_C})$, a contradiction to the Proposition 2.18 of \cite{OG}.
\end{proof}

\begin{remark}\label{IR1}
  Note that if all the sheaves in $\overline{X}$ are locally free then $\overline{X}=X$ and the restriction of the sheaves in $X$ to $C$ are locally free.
\end{remark}
Set $X_C = \{[E] \in \overline{X}: E_{\mid_C} \text{ is not semistable } \}$.  By \cite[Proposition 1.13]{OG}, $\text{dim}(X_C) \ge 4c_2 -39 -10= 4c_2 -49$.
Let $E \in X_C$. Since $E_{\mid_C}$ is not semistable, we have a destabilizing sequence,
\begin{equation}\label{A}
0 \to L \to E_{\mid_C} \to Q \to 0,
\end{equation}
where $L$ and $Q$ are line bundles on $C$ and we have $\text{deg}(L) > \mu(E_{\mid_C})=3$.

 Then there is a  stratification of the closed subset $X_C:= \{E \in \overline{X}: E_{\mid_C} \text{ is not semi-stable }\}$ as $X_C = \cup X_{C,d}$, where \\
 $X_{C, d}: = \{E \in X_C:  \text{ such that the degree of the line subbundle in the maximal destabilizing sequence is } d\} $. Clearly, $X_{C, d}$ is non-empty only if $d \ge 4$. We will prove that, under certain hypothesis,  $X_{C, d} $
 is non-empty for $4 \le d \le 7$. 
 \begin{lemma}\label{LA}
 If $H^0(C, L \otimes Q^*) =0$, where $L$ and $Q$ are as in exact sequence \ref{A}, then any non-trivial extension of the form,
 \[
 0 \to L \to V \to Q \to 0
 \]
 is simple.
 \end{lemma}
 \begin{proof}
 On the contrary, suppose $V$ is not simple. Then we have an endomorphism,
 $\phi: V \to V$ with $\text{ker}(\phi)$ and $\text{Im}(\phi)$ are  non-zero. If $\phi^2 \ne 0$, then  $V$ is decomposable as $V =\text{ker}(\phi) \oplus \text{Im}(\phi)$ and one could see that the extension can not be non-trivial.  Thus $\phi^2 = 0$. In other words, $\phi$ is nilpotent. Now  we have the following diagrams:

 \begin{equation}\label{C}
 \xymatrix{
 0 \ar[r] & L \ar[r] & V \ar[r]  \ar[d] & Q \ar[r] & 0 \\
  0 \ar[r] & \text{ker}(\phi) \ar[r] & V \ar[r] & \text{Im}(\phi) \ar[r] & 0.
 }
 \end{equation}  
 and
 \begin{equation}\label{B}
 \xymatrix{
 0 \ar[r] & \text{Im}(\phi) \ar[r] & V \ar[d] \ar[r] & \text{ker}(\phi) \ar[r] & 0 \\
 0 \ar[r] & L \ar[r] & V \ar[r] & Q \ar[r] & 0.
 }
 \end{equation} If $\text{deg}(\text{ker}(\phi)) < \text{deg}(Q)$ then 
 $\text{deg}(\text{Im}(\phi)) > \text{deg}(L)$  and  $\text{Hom}(\text{Im}(\phi), Q)= 0$ . Thus, from diagram \ref{B}we have,  $ \text{Hom}(\text{Im}(\phi), L) \ne 0$, a contradiction as $\text{deg}(\text{Im}(\phi)) > \text{deg}(L)$.  Therefore, $\text{deg}(\text{ker}(\phi)) \ge \text{deg}(Q)$. Similarly, if  $\text{deg}(\text{ker}(\phi)) > \text{deg}(Q)$, then $\text{deg}(\text{Im}(\phi)) < \text{deg}(L)$ and. Thus $\text{Hom}(L, \text{Im}(\phi))= 0= \text{Hom}(\text{ker}(\phi), Q)$. Thus we get a non-zero morphism $\text{ker}(\phi) \to L$ and a non-zero morphism $L \to \text{ker}(\phi)$. Thus  we have $L= \text{ker}(\phi)$. Hence $\text{Im}(\phi)= Q$.  Then from the diagram \ref{B},  we have either $V$ is a trivial extension or $H^0(C, L \otimes Q^*) \ne 0$, a contradiction.   
  \end{proof}
\begin{proposition}\label{P2}
Suppose  there exist an $E \in X_C$ with a line subbundle $L$ of degree $d-1$ such that $H^0(C, L \otimes {(E/L)}^*) = 0$ and  $L$ destabilizes $E$.  Then  we have $X_{C, d}$ is non-empty. 
\end{proposition}
\begin{remark}
Note that if $L$ is  a line bundle of degree $d$   with $H^0(C,L \otimes {(E/L)}^*)=0$, then by Riemann-Roch,  we have $d \le 7$.
\end{remark}
\begin{proof}
 Since $L$ destabilizes $E$, $\text{deg}(L) = d \ge 4$.

First, we prove the Proposition for $d =5$. Let $E$ be a such bundle and $L$ be a line subbundle of degree $4$ satisfying the hypothesis of the Proposition. 
Then we have the following exact sequence
\[
0 \to L \to E_{\mid_C} \to Q \to 0.
\]
 On the contrary, let we suppose that $X_{C, 5}$ is empty. Then we have a map $X_C \to J^4(C)$ which takes a torsion free sheaf $E$ to the destabilzing subbundle of $E \vert_C$, where $J^4(C)$ denotes the Jacobian of degree $4$ line bundles on $C$. Thus for a general $L \in J^4(C)$, the dimension of the subvariety $X_{C, L}$ of bundles $E \in X_C$ with $L$ being the maximal destabilizing subbundle is $ \ge 4c_2-49 -10 = 4c_2-59$.  Let $\mathbb{P}(H^1(C, L \otimes Q^*))$ be the  space of
 non-trivial extension classes. Since $H^0(C, L \otimes {(E/L)}^*) = 0$, the dimension of $\mathbb{P}(H^1(C, L \otimes Q^*) )= 6$.
 Then we have a map 
 \[
 \alpha: X_{C, L} \to  \mathbb{P}(H^1(C, L \otimes Q^*))
 \]
 which takes $E$ to $E_{\mid_C}$. 
 Assume that $\mathcal{M}(2, H)$ is fine moduli space. Let $\mathcal{E}$ be a tautological sheaf on $S \times X_{C,L}$ parametrized by $X_{C,L}$.  Let $\tilde{\mathcal{E}}$ be the restriction of $\mathcal{E}$ to $C \times X_{C, L}$. Clearly, $\tilde{\mathcal{E}}$ is a vector bundle. 
 Let $\text{Det}(\tilde{\mathcal{E}})$ denote the determinant line bundle  of the family
  $\tilde{\mathcal{E}}$ . Let $M$ be a line bundle on $C$ such that $\chi(\tilde{\mathcal{E}}_{\mid_{C \times \{x\}}} \otimes P_C^*M) = 0$ for all $x \in X_{C, L}$, where $P_C$ denote the projection onto $C$. Then  Grothendieck -Riemann-Roch gives
 \[
  \text{ch}((P_{X_{C, L}})_!(\tilde{\mathcal{E}} \otimes P_C^*M)) = (P_{X_{C, L}})_*(\text{ch}(\tilde{\mathcal{E}} \otimes P_C^*M).\text{Td}(C)),
 \]
 where $\text{Td}(C)$ denotes the relative Todd class. Considering the degree one component we have,
 
 \begin{equation}\label{A1}
 \text{Det}(\tilde{\mathcal{E}})=-c_1(\text{det}((P_{X_{C, L}})_!(\tilde{\mathcal{E}} \otimes P_C^*M)) )= (P_{X_{C, L}})_*(c_2(\tilde{\mathcal{E}}) - \frac{1}{4}c_1^2(\tilde{\mathcal{E}})).
 \end{equation}
 On the other hand, there is a natural Poincare extension on 
 $C \times \mathbb{P}(H^1(C, L \otimes Q^*))$:
 \[
 0 \to P_C^*L \otimes P_{\mathbb{P}(H^1(C, L \otimes Q^*))}^* \mathcal{O}(1) \to \mathcal{F} \to P_C^*Q \to 0.
 \]
Now consider the  morphism
 $\Phi= \text{Id} \times \alpha: C \times X_{C, L} \to C \times  \mathbb{P}(H^1(C, L \otimes Q^*))$ . Since by Lemma \ref{LA},  $\mathcal{F}_{\mid_{C \times \{y\}}}$ is simple for all $y \in \mathbb{P}(H^1(C, L \otimes Q^*))$, 
 \[
 0 \subset P_C^*L \otimes P_{\mathbb{P}(H^1(C, L \otimes Q^*))}^* \mathcal{O}(1) \subset \mathcal{F}
 \]
 is the HN-filtration of $\mathcal{F}$. Hence   we have $\Phi^*\mathcal{F} \cong \tilde{\mathcal{E}}$.
 Therefore, the determinant bundle $\text{Det}(\tilde{\mathcal{E}})$ of the family $\tilde{\mathcal{E}}$ of rank-two sheaves on $C$, is the 
 pullback of the determinant bundle $\text{Det}(\mathcal{F})$ of the family $\mathcal{F}$. In other words,
 we have 
 \begin{equation}\label{A2}
 \text{Det}(\tilde{\mathcal{E}})=-c_1(\text{det}((P_{X_{C, L}})_!(\tilde{\mathcal{E}} \otimes P_C^*M))) \cong \Phi^*(-c_1(\text{det}((P_{\mathbb{P}(H^1(C, L \otimes Q^*))})_!(\mathcal{F} \otimes P_C^*M))))
 =\text{Det}(\mathcal{F})
  \end{equation}
 Choose an integer $k$ to be sufficiently large. It is known that for a smooth curve in $|kH|$ for sufficiently  large $k$, the restriction of a stable bundle remains stable. Also by Serre's vanishing Theorem, if $k >>0$, then for all $E_1, E_2 \in X_{C, L}$, we have 
 \[
 H^1(X, E_1^* \otimes E_2 \mathcal{O}(-kH)) = 0.
 \]
 Therefore the restriction map 
 \begin{equation}\label{A3}
 \rho: X_{C, L} \to \mathcal{M}(C_k, H_{\mid_{C_k}})
 \end{equation}
 is injective. 
   Let $\tilde{\mathcal{E}}_k$ be the restriction of the family $\mathcal{E}$ to $C_k \times X_{C, L}$.
 Let $M_k$ be a line bundle on $C_k$ such that $\chi((\tilde{\mathcal{E}}_k \otimes P_{C_k}^*M_k)_{\mid_{C_k \times \{x\}}}) = 0$ for all $x \in X_{C, L}$.  Then as in \ref{A1}, 
 we have,
 \begin{equation}\label{A4}
 -c_1(\text{det}((P_{X_{C, L}})_!({\tilde{\mathcal{E}}}_k \otimes P_C^*M_k))) = (P_{X_{C, L}})_*(c_2({\tilde{\mathcal{E}}}_k) - \frac{1}{4}c_1^2({\tilde{\mathcal{E}}}_k)).
 \end{equation}
 Then the left-hand side of the equality in \ref{A4} is
identified with the first Chern class of $\rho^* \Theta_k$ (see \cite{DN}), where $\Theta_k$ is the theta divisor on $\mathcal{M}(C_k, H_{\mid_{C_k}})$, while the right-hand side equals the slant
product 
 \[
 (c_2(\mathcal{E} )-\frac{1}{4}c_1^2(\mathcal{E}))/[kH].
 \]
 Therefore we have, 
 \begin{align*}
 \begin{split}
 c_1(\rho^*(\Theta_k)) = &k \text{Det}(\tilde{\mathcal{E}})\\
 =  & k\Phi^*\text{Det}(\mathcal{F}).
 \end{split}
 \end{align*}
 Since $\rho$ is injective and $\Theta_k$ is ample, $\rho^*(\Theta_k)$ is also ample.  Therefore, we have $\Phi^*\text{Det}(\mathcal{F})$ is ample and hence $(\Phi^*\text{Det}(\mathcal{F}))^{\text{dim}(X_{C, L})} >0$. But if the dimension of $X_{C, L} > 6$, then $(\Phi^*\text{Det}(\mathcal{F}))^{\text{dim}(X_{C, L})}= 0$, a contradiction.
 This proves the proposition for $d=5$, under the assumption that there is a tautological sheaf $\mathcal{E}$ on $S \times X$.
 In general, by Theorem (A.5) in \cite{Mukai} there exists a quasi-tautoloqical
sheaf, $\mathcal{G}$ on $S \times X$, i.e. such that for $[E] \in X$, the restriction $\mathcal{G} \vert{S \times \{[f]\}}$
is isomorphic to $F^{\oplus k}$ for some positive integer $k$. One can repeat the proof
above with $\mathcal{E}$ replaced by $\mathcal{G}$.
 
Inductively, let $X_{C, d}$ be non-empty for $4 \le d < 7$.  Then by \cite[P. 64, 6.37]{OG} the codimension of $X_{C, d}$ in $X$ is $2d+g -7$. Thus, the dimension of $X_{C, d} = 4c_2 -39 -2d -10 +7 = 4c_2-42-2d$. Again, let $E$ be a point in $X_{C, d}$ which sits in the exact sequence,
\[
0 \to L \to E_{\mid_C} \to Q \to 0
\]
where $L, Q$ are line bundles on $C$ with $H^0(C, L \otimes Q^*) = 0$. Then we have the dimension of $X_{C, L}$ is $4c_2 -52 -2d$. On the other hand, the dimension of the projective space $\mathbb{P}(H^1(C, L \otimes Q^* ))$ is $14-2d$. Therefore, we can repeat the argument as long as $4c_2 -52 -2d > 14-2d$ and $14-2d \ge 0$. Now $14-2d \ge 0 \Leftrightarrow d \le 7$ and $4c_2 -52 -2d > 14-2d \Leftrightarrow c_2 \ge 17$.   
      \end{proof} 
     Now  we have the following Corollary. 
 \begin{corollary}\label{CL1}
 Suppose  $X_{C, 5}$ is non-empty  and there exist an $E \in X_{C, 5}$ with a destabilizing  line subbundle $L$  satisfying  $H^0(C, L \otimes {(E/L)}^*) = 0$, then $X_{C, 6}$ is non-empty.
 \end{corollary}

\section{Boundary of the moduli space}\label{S3}

 For a  subset $X$ of the moduli space $\overline{\mathcal{M}}(H, c_2)$, define 
the boundary of $X$ as, 
\[
 \partial X = \overline{X} \setminus \overline{X} \cap \mathcal{M}(H, c_2),
\]
where $\overline{X}$ denotes the closure of $X$ in $\overline{\mathcal{M}}(H, c_2)$.
Note that if $X$ consists only locally free sheaves, then the boundary of $X$ is
\[
 \partial X= \{[F] \in \overline{X}: F \text{ is not locally free} \}.
\]

In this section we shall prove that if $c_2 \ge 34$ then any irreducible good component  $X \subset \overline{\mathcal{M}}(H, c_2)$ 
 has non-empty boundary. 

\begin{theorem}\label{t1}
 Let $X \subset \mathcal{M}(H, c_2)$ be an irreducible good component. If $c_2 \ge 34$ then $\partial X \ne\varnothing$.
 \end{theorem}
 \begin{proof}
Let us assume that $\partial X = \varnothing$. Then $\overline{X}=X$ and hence by remark \ref{IR1} we can apply the results from the section \ref{S2}.\\
Let $C$ and $X_C$ be as in Section \ref{S2}.
 By Proposition \ref{P1}, $X_C  \ne \varnothing$.

Let $E$ be a vector bundle in $X_C$ and $U$ be an irreducible neighbourhood of $E$ in $X$. Let $U^{ns} \subset U$ be the subset of $U$
consisting of vector bundles $E$ which are not semistable when restricted to $C$.
Then $U^{ns}$ has codimension  $g_C$ 
\cite[Proposition 1.13]{OG}. Thus $\text{dim} U^{ns} \ge 4c_2-39-10 $. On the other hand $\text{dim }V \le 3c_2-21$. Thus,  
for $c_2 \ge 29$ a  general element  $F \in X_C$ has  $h^0(S, F) = 0$.   Thus, by Proposition \ref{Prop1}, a general element in $X_C$ is either smooth or it is potentially obstructed of type II. 
But the dimension of the locus of potentially obstructed bundles of type II, is at most 39.   Therefore, a general element in $X_C$ has $H^0(S, \mathcal{E}nd^0F(2)) = 0$.

Let $F$ be a such  vector bundle, i.e.,   $H^0(S, F)= 0 ,  F_{\mid_C}$ is not stable  and 
 $H^0(S, \mathcal{E}nd^0F(2)) = 0$.
Now  we have a destabilizing sequence  
\[
 0 \longrightarrow L \longrightarrow F \vert_C \longrightarrow Q \longrightarrow 0,
\]
where $L, Q$ are line bundles on $C$ .

Consider the elementary transformation, 
\[
 0 \longrightarrow E \longrightarrow F \longrightarrow \iota_*Q \longrightarrow 0,
\]
where $\iota: C \longrightarrow S$ is the natural injection. Note that $c_1(E)= c_1(F) - [C]= 0$. Now, $H^0(S, F)= 0$ and
$\text{Pic}(S) = \mathbb{Z},$
implies that $E$ is stable. Since $L$ occurs in the above destabilizing sequence,  degree of $L \ge 4$ and 
hence degree of $L \otimes Q^* \ge 2$.

Restricting this sequence to $C$ one gets
\[
 0 \longrightarrow Q(-C) \longrightarrow E \vert_C \longrightarrow L \longrightarrow 0
\]
where $Q(-C)\,=\,Q \otimes (\mathcal {O_S(-C)}_{\mid_C})$. 
Let $Y_F:= Quot(E_{\mid_C}, L)$ be the Grothendieck Quot-scheme parametrizing quotients of $E_{\mid_C}$, that have the same Hilbert
polynomial as $L$.   This set parametrizes a family $\{F_y\}_{y \in Y_F}$ of torsion free sheaves with 
$c_1(F_y)= H$ and 
$c_2(F_y)= c_2$ for all $y \in Y_F$, where $F_y = G_y(C)$ for a subsheaf $G_y$ of $E$, which
is defined as the kernel in the exact sequence:
\[
 0 \longrightarrow G_y \longrightarrow E \longrightarrow i_*L_y \longrightarrow 0.
\]
{\bf Claim 1}: $G_y$ is stable for all $y \in Y_F$.\\
{\bf Proof of claim 1}:
Note that $c_1(G_y) = -H$. Therefore $\mu(G_y) = -3$. If $G_y$ is not stable then it has a line subbundle of the form 
$\mathcal{O}_S(k)$, where $k \ge 0$. If $k = 0$, then $G_y$ has a non-zero section, which gives a non-zero section of
$E$ . Therefore, $F$ admits a non-zero section, which
is a contradiction. 

Thus, we obtain a map
\[
 \varphi: Y_F \longrightarrow \bar{\mathcal{M}}(H, c_2)
\]
and therefore a subset $\varphi^{-1}(\overline{X}) \subset Y_F$. Since $F$ is a smooth point, $h^0(\mathcal{E}nd^0F(2))=0$, thus by the following  Lemmas, we have $\text{dim}(\varphi^{-1}(\overline{X})) =  \text{dim}(Y_F)$.
From the definition of Quot scheme we have 
\begin{align*}
\begin{split}
\text{dim}(Y_ F)  & = h^0(Q^{*}(C) \otimes  L) \\
 & \ge \text{deg}( L) -\text{deg}( Q) + \text{deg}(\mathcal O_{C}(C) )- 9.
 \end{split}
 \end{align*}
 Clearly, if  $\text{deg}(L) \ge 6$, then we have $ \text{dim}(\varphi^{-1}(\overline{X}))  > 1$.\\
 If $\text{deg}(L) \le 5$, then consider  two cases:\\
{\bf Case I}: $H^0(C, L \otimes Q^*) \ne 0$.\\
From the exact sequence,
\[
0 \to \mathcal{O}_S \to \mathcal{O}_S(C) \to \mathcal{O}_C(C) \to 0
\]
we have $h^0(C, \mathcal{O}_C(C)) =3$. Since $H^0(C, L \otimes Q^*) \ne 0, h^0(C, L \otimes Q^* \otimes \mathcal{O}_C(C)) \ge 3$.  Hence $\text{dim}(Y_F) \ge 3$.

{\bf Case II}:$H^0(C, L \otimes Q^*) = 0$.

If $\text{deg}(L) = 4$, then by Proposition \ref{P2}, we have $X_{C, 5}$ is non-empty. Let  $E \in X_{C, 5}$. If  the destabilizing line subbundle $L$  of $E$ has $H^0(C, L \otimes Q^*) \ne 0$, then as in case I, we have $\text{dim}(Y_F) \ge 3$. Otherwise, by Corollary \ref{CL1}, $X_{C, 6}$ is non-empty. 
If $\text{deg}(L) = 5$, then also by Corollary \ref{CL1}, $X_{C, 6}$ is non-empty.  But we have  $\text{dim}(X_{C, 6}) = 4c_2- 54$. By Proposition \ref{p2}, the vector bundles with a non-zero section has dimension $3c_2-21$, Thus if $c_2 \ge 34$ , then we can choose a vector bundle $F \in X_{C, 6}$ with
$h^0(S, F) = 0$ and we have  $\text{deg}(L) = 6$. Thus the dimension of $Y_F \ge 3$. 
Therefore, we have,  $\text{dim}(\varphi^{-1}(\overline{X})) >1$
  Hence by \cite[Lemma 2.15]{OG}, we conclude the 
existence of a boundary point in $\partial X$. 

\end{proof}
\begin{lemma}\label{l1}
 $\text{dim} (\varphi^{-1}(\overline{X})) = \text{dim}(Y_F)$, where $\mathcal{E}nd^0F$ denotes the traceless endomorphisms of $F$.
 \end{lemma}
 \begin{proof}
 The morphism $\varphi$ is an immersion, hence $\varphi^{-1}(\overline{X})\,=\, \overline{X} \cap Y_ F$ and we have $\text{dim} (\varphi^{-1}(\overline{X})) \le \text{dim}(Y_F)$. Hence by affine dimension count   and the fact that 
 $\text{dim}(T_F\mathcal{M}(H, c_2)) \,\ge \, \text{dim} (\mathcal{M}(H, c_2))$ we have
\[
\text{dim}( \varphi^{-1}(\overline{X})) \,=\,\text{dim}(\overline{X} \cap Y_F) \ge \text{dim }\overline{X} + \text{dim }Y_F - \text{dim }T_{[F]}\mathcal{M}(H, c_2).
\]
It is known that $\text{dim }T_{[F]}\mathcal{M}(H, c_2)-\text{dim }X$ is bounded above by 
$h^2(\mathcal{E}nd^0F)= h^0(\mathcal{E}nd^0F(2))=0$, which concludes the Lemma.

  \end{proof}
  
 Now we include a Proposition in this section which we need in the next section.
  \begin{proposition}\label{LMM1}
 Let $S$ be an irrducible smooth projective surface over complex numbers. Let $\mathcal{F}$ be a family of rank $2$ torsion free 
 sheaves on $S$ parametrized by an irreducible scheme $B$. Let $B_0 \subseteq B$ be the subset of points $b$ such that 
 $\mathcal{F}_b$ is not locally free at a zero-dimensional subscheme of length at least $2$. Assume $B_0$ is non-empty. Then
 it has codimension at most $4$in $B$.
\end{proposition}
 
 \begin{proof}
 The proof goes in the same line argument as in \cite{OG2}.
  Let $\text{Sing}(\mathcal{F})$ be the set of points $x \in S \times B$ such that the stalk of $\mathcal{F}$ at $x$ is not
   free.
  Let $b_0 \in B_0$ be such that $\mathcal{F}$ is not free at $p_1= (x_1, b_0)$ and $p_2=(x_2, b_0)$. 
  Since $S$ is projective, there exists an affine neighborhood containing both $p_1$ and $p_2$. Let $\text{Spec}(A)$ be such an affine
  neighborhood, then $\mathcal{F}(\text{Spec}(A)) =M$, where $M$ is 
  a torsion free $A-$ module of rank 2. Consider a short free resolution
  \[
   0 \longrightarrow A^n \longrightarrow^f A^{n+2} \longrightarrow M \longrightarrow 0.
  \]
Evaluating $f$ at $p_1$ and $p_2$ we get two morphism 
$\alpha_i: \text{Spec}(A_{p_i}) \longrightarrow \text{Hom}(k^n, k^{n+2}), i= 1, 2$.

Let $\text{Hom}_1 \subset \text{Hom}(k^n, k^{n+2})$ be the subset of maps with positive dimensional kernel.
Then Clearly 
\[
 \text{Sing}(\mathcal{F}(\text{Spec}(A_{p_i}))) = \alpha_i^{-1}(\text{Hom})_1, i= 1, 2.
\]
But 
\[
 \text{codim}(\text{Hom}_1, \text{Hom}(k^n, k^{n+2})) = 3.     
 \]
On the other hand,  codimension of 
$S \times B_0 \le \text{codim } \text{Sing}(\mathcal{F}(\text{Spec}(A_{p_2})))\cap \text{Sing}(\mathcal{F}(\text{Spec}(A_{p_2})))\le 3+3=6.$
There for the codimension of $B_0$ in $B$ is at most $4$.
\end{proof}
\begin{remark}\label{RZZ}
In fact, repeating the same argument one can inductively prove that if $B_0$ consists of points which are not locally free at $m$ points and $B_0$ is non-empty, then co-dimension 
of $B_0$ in $B$ is at most $3m-2$.
\end{remark} 



\section{connectedness of the moduli space}\label{S4}
In this section, we will prove that the moduli space $\overline{\mathcal{M}}(H, c_2)$ is connected for $c_2 \ge 27$. First we prove some
Lemmas which we need to prove the connectedness theorem.

\begin{lemma}\label{LM2}
 $\overline{V}= \overline{V}(c_2):= \{[F] \in \overline{\mathcal{M}}(H, c_2) | h^0(F) \ne 0\}$ is connected for $c_2 \ge 27$.
\end{lemma}
 \begin{proof}
   Note that,  by remark \ref{RZ}, $\overline{\mathcal{M}}(H, c_2)$ is good for $c_2 \ge 27$. Hence $\mathcal{M}(H, c_2)$ is dense in $\overline{\mathcal{M}}(H, c_2)$. Therefore, 
   $ \overline{V}$ is the closure of $V(c_2)$, where $V(c_2)$ consists only locally free sheaves with non-zero sections.\\
    Let $F \in \overline{V}$. Since the maximal slope of subsheaves is zero, we have a short exact sequence 
 \[
 0 \rightarrow \mathcal O \rightarrow F \rightarrow \mathcal{J}_{Z}(1) \rightarrow 0.
 \]
 
 Let $\mathcal{N}$ be the space of pairs 
\[
 \mathcal{N}= \{(E, s): E  \in \overline{V},  s \in \mathbb{P}(H^0(S, E))\},
\]
as in Proposition \ref{p1}.
Consider the following diagram
 \[
\xymatrix{
& \mathcal{N} \ar[r]^{p_2} \ar[d]^{p_1} & \text{Hilb}^{c_2}(S)\\
& \overline{V}\\
}  
\]
Clearly $p_1$ is surjective.  
Since $c_2 \ge 20$ and $H^0(S, \mathcal{O}_S(3)) = 20$, for a general point in $\text{Hilb}^{c_2}(S)$, we have  $h^0(\mathcal{J}_{Z}(3)) = 0$.  Thus, if $c_ \ge 21$, a general point in $\text{Hilb}^{c_2}(S)$ satisfies Caley-Bacharach property for cubics. Therefore, for a general point $Z \in \text{Hilb}^{c_2}(S)$, the extension group $H^1(S, \mathcal{J}_Z(3))$ is nontrivial.  Hence, $p_2(\mathcal{N})$ is a dense subset of $\text{Hilb}^{c_2}(S)$ and hence connected.\\ 
  Since $p_2(\mathcal{N})$ is connected and the fibers over the projection map $p_2$    are the projective spaces $\mathbb{P}(\text{Ext}^1(\mathcal{J}_Z(1),\mathcal O)), \mathcal{N}$ is connected.  Therefore, $\overline{V} $ is also connected. 

 \end{proof}

Let $c_2 \ge 27$ and $X$ be an irreducible component  of $\overline{\mathcal{M}}(H, c_2)$ which is contained in a connected component  $Z$ not intersecting the connected component containing  $\overline{V}(c_2)$. Further, assume that the 
boundary $\partial{X}$ is non-empty. Note that for  $c_2 \ge 27, \overline{\mathcal{M}}(H, c_2)$ is good. Thus, there is no component of $\overline{\mathcal{M}}(H, c_2)$ consisting only boundary elements. Thus, by  \cite[Prop. 4.3]{OG}, $\partial{X}$ is of codimension 1.
Let $\mathcal{M}(c_2, c_2^\prime):= \{[F] \in \overline{\mathcal{M}}(H, c_2)| F \text{ is not locally free with } c_2(F^{**})=c_2^\prime \}$.
\begin{remark}\label{R2}
  Note that a general element in $\mathcal{M}(c_2, c_2^\prime)$ is a smooth point of $\overline{\mathcal{M}}(H, c_2)$ for $c_2^\prime \ge 20$. To see this, 
  let $F$ be a general torsion-free sheaf.
Then it is the kernel of a general surjection $E \longrightarrow S$ from a stable bundle $E$ general in $\mathcal{M}(H, c_2^\prime)$
 to a sheaf $S$ of length $c_2 - c_2^\prime$. 
 
 If $F$ were a singular point, then there would exist a nontrivial co-obstruction $\phi: F \longrightarrow F(2)$. This would
 have to come from a nontrivial co-obstruction $E \longrightarrow E(2)$ for $E$, but that can not exist as a general point $E$ 
 is smooth in  $\mathcal{M}(H, c_2^\prime)$ for $c_2^\prime \ge 20$.
 
 On the other hand, if $c_2^\prime \le 18$ then a general element  $F \in \mathcal{M}(H, c_2)$ we have $h^0(F) \ne 0$,  \cite[Theorem 4.1]{DS}. 
 For $c_2 = 19$, if all the cohomologies  of $F$ vanishes, then a general such point is smooth. If one of the cohomologies do not vanish then, as in \cite[Theorem 4.1]{DS},  one can easily see that a general element $F$, $h^0(F) \ne 0$. 
 \end{remark}

\begin{proposition}\label{LM3}
 Suppose $Z$ contains a sheaf which is not locally free along a subscheme of length at least $m-1$, where $m$ is an integer satisfying $4c_2 -39 -(3m-2)-(2m-2)> 27$.  Then $Z$ contains a sheaf which is not locally free along a subscheme of length  at least $m$. 
  \end{proposition}
\begin{proof}
 On the contrary, let assume that   all sheaves in the boundary $\partial Z$  of $Z$ are not locally
free along a subscheme of length  at most $m-1$. 

Let $(\partial Z)_{m-1}(S)$ be the set which  parametrizes torsion-free sheaves  $F \in \partial Z$ which fit in the following exact sequence \[
0 \to F \to F^{**} \to Q_F \to 0,
\]
where $Q_F$ is a coherent sheaf of length $m-1$. By hypothesis. $\partial{Z}_{m-1}(S)$ is non-empty. \\
Let $C \in |H|$ be a smooth curve.\\
{\bf Claim I}:\\
There is a subscheme $P$ of length $m-1$ supported at $m-1$ distinct points not intersecting $C$ and a closed subset of $\partial{Z}$ of dimension $\ge 28$ such that every sheaf in the closed set is not locally free along $P$.\\
{\bf Proof of the Claim I}:\\
 By \cite[Proposition 6.4]{Li2}, a general element in $\partial{Z}_{m-1}(S)$ is not locally free along a reduced subscheme of length ${m-1}$. Thus we have a rational map:\\
$\sigma: \partial{Z}_{m-1}(S) \dashrightarrow  S^{m-1}-\text{ diag}$ (the space of choices of distinct (m-1)-uple of points in
S),\\
which takes $F \to \text{support}(Q_F)$.
Let $P$ be a general element in the image of $\sigma$ which do not intersect $C$. Then $\text{dim}(\sigma^{-1}(P)) \ge \text{dim}(\partial{Z}_{m-1}(S))-(2m-2) $.
Note that, the set of sheaves in $\partial{Z}$, which are not locally free along $P$ is closed in $\partial{Z}$. Thus $\sigma^{-1}(P)$ is a closed subset of $\partial{Z}$.
By remark \ref{RZZ},  the dimension of $(\partial Z)_{m-1}(S) \ge \text{dim}(Z) -(3m-2) \ge 4c_2-39 -(3m-2)$.\\
Thus $\text{dim}(\sigma^{-1}(P)) \ge 4c_2-39-(3m-2)-(2m-2)\ge 28$, by hypothesis.\\

Since $P$ does not intersect $C$, every sheaf in $\sigma^{-1}(P)$ is locally free when restricted to $C$.\\
{\bf Claim II}:\\
There is a sheaf $F \in \sigma^{-1}(P)$ such that $F \vert_C$ is not semistable.\\
{\bf Proof of the Claim II}:\\
On the contrary let us assume that $E \vert_C$ is semistable for all $E \in \sigma^{-1}(P)$.
Then, we have a morphism $\rho: \sigma^{-1}(P) \to \mathcal{M}(C, H \vert_C)$.
    Since the dimension of $\text{dim}(\sigma^{-1}(P)) \ge 28$,   arguing as in  Proposition \ref{P1} in Section \ref{S2}, one can show that 
 there exists a stable  torsion free sheaf  which  is nonstable when restricted to $C$.\\
 Let $F$ be a general such element. Note that if $c_2 \le 18$, then a general element has a non-zero section. We assume $c_2 \ge 19$.  By remark \ref{R2},  we can further assume $F$ to be smooth. \\
 Now we  have a destabilizing sequence of the form,
 \begin{equation}\label{*}
 0 \to L \to F_{\mid_C} \to Q \to 0.
 \end{equation}
 
 
 Consider the elementary transformation 
 \[
 0 \to E \to F \to \iota_*Q \to 0.
 \]
 Restricting this sequence to $C$, one gets 
 \[
 0 \to Q(-C) \to E \vert_C \to L \to 0.
 \]
 Let $Y_F:= \text{Quot}(E \vert_C, L)$ be the Grothendieck Quot-scheme parametrizing quotients of $E \vert_C$, that have the same Hilbert polinomial as $L$.
 $Y_F$ parametrizes 
  a family $\{F_v\}_{v \in Y_F}$ of $\mu$-stable sheaves with 
$c_1(F_v)= H$ and 
$c_2(F_v)= c_2$ for all $v \in Y_F$, where $F_v = G_v(C)$ for a subsheaf $G_v$ of $E$, which
is defined as the kernel in the exact sequence:
\begin{equation}\label{eqz}
 0 \longrightarrow G_v \longrightarrow E \longrightarrow i_*L_v \longrightarrow 0.
\end{equation}
Let $\pi_C^*(E \vert_C) \to \mathcal{L}$ be the tautological quotient sheaf on $C \times Y_F$ and let $\mathcal{G}$ be the sheaf on $S \times Y_F$ fitting in to the exact sequence,
\[
0 \to \mathcal{G} \to \pi_S^*(E \vert_U) \to {(\iota \times \text{Id})}_*\mathcal{L} \to 0.
\]
As in the proof of \cite[Lemma 2.13]{OG}, one can see 
 that this is a flat family over $Y_F$ and $G_v$ is not locally free at $y \in S \setminus P$ if and only if $L_v$ is not
locally free.
Thus, we obtain a map
\[
 \varphi: Y_F \longrightarrow \overline{\mathcal{M}}(H, c_2)
\]
and a subset $\Sigma:= \varphi^{-1}(\sigma^{-1}(P)) \subset Y_F$. As in Theorem \ref{t1}, we have  $\text{dim}(\Sigma) \ge 2$.  Now the proposition follows from the following Lemma.
\end{proof}
\begin{lemma}\label{LM4}
Let  $\Sigma \subset Y_F$ be a closed subvariety. If $\text{dim}(\Sigma) \ge 2$, then $\Sigma$ contains a point which is not locally free at, at least $m$ points.
\end{lemma}

\begin{proof}
Let $P =\{x_1, x_2, ..., x_{m-1}\}$.
Recall that  $\Sigma$ is a family of sheaves which are not locally free at  fixed  $m-1$ points $x_1, x_2,..., x_{m-1}$ which are  not on the curve $C$. Thus, all sheaves in $\Sigma$ are locally free on $C$.  We prove the Lemma by contradicting the fact 
  that all sheaves of the family are locally free along $C$. 
  
  If  all sheaves in $\Sigma$ are locally free except at $x_1, x_2, ..., x_{m-1}$, then the line bundles $L_v, v \in \Sigma$ in \ref{eqz} are locally free except at $x_1, x_2, ..., x_{m-1}$.
Thus, for each $v \in \Sigma$ we get an exact sequence of locally free sheaves of the form,
\[
 0 \longrightarrow Q_v \stackrel{h_v} \longrightarrow  E \mid_C \longrightarrow L_v \longrightarrow 0,
\]
which induces a morphism $\theta$ from $\Sigma$ to the Hilbert scheme parametrizing subvarieties of $\mathbb{P}(E_{\mid_C})$, namely, 
$ v \mapsto \bar{h_v}(\mathbb{P}(Q_v))$, where $\bar{h_v}$ is induced by $h_v$. Clearly $\theta$ is injective. Hence $\text{dim} \theta(\Sigma) \ge 2$.
Now choose a point $P \in C$ and let $E_P$ be the fiber of $E$ at $P$. Then we have a morphism 
 $ \rho: \Sigma \longrightarrow \mathbb{P}(E_P)$ associating to $v \in \Sigma$, the point $\mathbb{P}(Q_{v,P}) \in \mathbb{P}(E_P)$ which 
 factors through $\theta$. Since $\Sigma$ has
dimension $\ge 2$, $\text{dim}(\rho^{-1}([V])) \ge 1$, where $[V]$ is a point in the image of $\rho$.
Let $\Omega \subset \mathbb{P}(E_{\mid_C})$ be defined by 
\[
 \Omega:= \cup_{v \in \rho^{-1}([V])} \bar{h_v}(\mathbb{P}(Q_v)).
\]
Since $\theta$ is injective, $\text{dim}(\Omega) \ge 2$. Since $Y_F$ is complete and so is $\Sigma$, $\Omega$ is a closed subvariety of $\mathbb{P}(E_{\mid_C})$.
Thus, we have $\text{dim}(\Omega \cap \mathbb{P}(E_P)) \ge 1$. Thus, we have a complete family of dimension at least one of quotient sheaves
locally free along $C$, that
are all the same quotient at a point $P$, which is a contradiction as $\text{dim}(\mathbb{P}(V)) = 0$.
\end{proof}

\begin{corollary}\label{CL1}
 Let $Z$ as in Proposition \ref{LM3}. Then for $c_2 \le 27,  \mathcal{M}(c_2, d) \cap \partial{Z}$ is non-empty for some $d \le 18$. 
  \end{corollary}
 \begin{proof}
Take $m = c_2 -18$.  Note that for $c_2 \le 27, 4c_2 -39 -3(c_2-18) +2 -2(c_2-18)+2 = 55-c_2 \ge 28 $.  Since $\partial Z$ is non-empty, using the Proposition \ref{LM3} inductively
 one can conclude the corollary.
 \end{proof}
 
 
\begin{proposition}
 For $c_2 \ge 27, \overline{\mathcal{M}}(H, c_2)$ is connected. 
 \end{proposition}
 \begin{proof}
Let $X$ be an irreducible component of $ \overline{\mathcal{M}}(H, c_2)$ which does not intersect the connected component 
  containing $\overline{V}$. 
Thus, $H^0(E) = 0$ for all $E \in X$ and since $c_2 \ge 27$, by Proposition \ref{P1}  $X$ contains a point $E$ which is non-stable when restricted to $C$. Thus, $E_{\mid_C}$ has a destabilizing sequence of the form,
\[
0 \to L \to E_{\mid_C}  \to Q \to 0.
\]
Then, as in proof of Theorem \ref{t1},
$X$ contains a boundary point  which is not locally free.\\ (Note that the hypothesis $c_2 \ge 34$ in the Theorem \ref{t1} is used to ensure a vector bundle in $X_{C, 6}$ which admits no nonzero section. But by the hypothesis on $X$, we have $h^0(E)= 0$ for all $E \in X$. Thus the arguments in Theorem \ref{t1}, works fine without the assumption $c_ \ge 34$.)\\
On the other hand, for $c_2 \ge 21$,  the connected component containing $\overline{V}$ has a non-empty boundary too. In fact,  for $Z$ being $c_2-1$ points on a cubic plus a point not on the cubic, $h^1(\mathcal{J}_Z(3)) \ne 0$, but $Z$ does not have the Cayley-Bacharach property with respect to cubics. Thus,  the torsion free sheaf defined by $Z$ is not locally free.  Therefore, all the connected components of $\overline{\mathcal{M}}(H, c_2)$  have non-empty boundaries. 

 Let $E$ be a point of $\overline{\mathcal{M}}(H, c_2)$ in the boundary.
 Then we have the following exact sequence 
 \[
  0 \to E \to E^{**} \to E^{**}/E \to 0.
  \]
 
 By, \cite[Theorem 4.2]{SIM}, $E$ may be deformed to a
sheaf $E^{\prime}$ which is a kernel as above, but with quotient  a direct sum
of length $1$ sheaves at distinct points of $S$.
Thus, we can assume the zero cycle 
$Z_{E}\,=\,\text{supp}{E^{**}/E}$
 consists of $d$ distinct points. If $c_2 \ge 28$, removing a single point from $Z_{E}$ 
gives a torsion 
free sheaf $G$ such that $E \supset G \subset E^{**}$ and $c_2(G)= c_2(E) -1$.
This gives  a point in  $ \overline{\mathcal{M}}(H, c_2-1)$.  Therefore, all the components of $\overline{\mathcal{M}}(H, c_2)$ gives a component in $\overline{\mathcal{M}}(H, c_2 -1)$.
In other words, the number of connected components of $\overline{\mathcal{M}}(H, c_2)$ is a decreasing function of $c_2$ for $c_2 \ge 27$.

This completes the proof of connectedness for $c_2 \ge 27$ by induction, if we know it for $c_2 = 27$.
Hence, it is enough to prove that $\overline{\mathcal{M}}(H, 27)$ is connected.

 Let $X$ be an irreducible component  of $\overline{\mathcal{M}}(H, 27)$ which is contained in a connected component  $Z$ not intersecting the connected component containing  $\overline{V}(27)$. Since by above arguments, every connected component of $\mathcal{M}(H, 27)$ has a non-empty boundary,  the 
boundary $\partial{Z}$ is non-empty. \\
Thus, by Corollary \ref{CL1},  $\mathcal{M}(27, d) \cap Z$ is non-empty for some $d \le 18$.  
Let $F$ be a general point in $\mathcal{M}(27, d) \cap Z$ for some $d \le 18$. Then we have $c_2(F^{**}) =d$. 
 By \cite[Theorem 4.1]{DS}, $h^0(F^{**}) \ne 0$.
 
 Thus, we have an exact sequence
\begin{equation}\label{Z1}
0 \to \mathcal{O} \to F^{**}  \to \mathcal{J}_P(1) \to 0
\end{equation}
where $P$ is a zero-dimensional sub-scheme of length $d$.  \\
  Let $W$ be a zero cycle   of length $l(F^{**}/F)$ disjoint from $\text{Supp}(F^{**}/F)$.    Consider the natural map, 
\[
\text{Ext}^1(\mathcal J_{P}(1), \mathcal{O} \rightarrow \text{Ext}^1(\mathcal J_{P+W}(1),\mathcal O),
\]
 let $G$ denote the image of the extension class $F^{**} $ under the above map. Then we have $G \hookrightarrow F^{**} $. Since $W$ can be obtained from a continuous deformation of the points in the support of $F^{**}/F$, 
 $G$ sits inside the same connected component of $\partial X$ as $F$. Note that by construction of $G, h^0(G) \ne 0$.
    Thus, we have connected $F$ to a point $G  \in \overline{V}(27)$.
 This shows that $\overline{\mathcal{M}}(H, 27)$ has only one connected component. 

\end{proof}



\section{irreducibility of the moduli space}
\subsection{Boundary strata}
The boundary $\partial{\mathcal{M}(H, c_2)}: = \overline{\mathcal{M}}(H, c_2) - \mathcal{M}(H, c_2)$ is the set of points 
corresponding to torsion-free sheaves which are not locally free.

Let $\mathcal{M}(c_2, c_2^\prime):= \{[F] \in \overline{\mathcal{M}}(H, c_2)| F \text{ is not locally free with } c_2(F^{**})=c_2^\prime \}$.
Then the boundary has a decomposition into locally closed subsets 
\[
\partial{\mathcal{M}(H, c_2)}= \amalg_{c_2^\prime < c_2} \mathcal{M}(c_2, c_2^\prime).
\]
By the construction of $\mathcal{M}(c_2, c_2^\prime)$, we have a well defined map,
\[
 \mathcal{M}(c_2, c_2^\prime) \longrightarrow \mathcal{M}(H, c_2^\prime).
\]
The map  takes $E \longrightarrow E^{**}$. The fiber over $E \in \mathcal{M}(H, c_2^\prime)$ is the Grothendieck Quot-scheme
$\text{Quot}(E; d)$ of quotients of $E$ of length $d: = c_2 -c_2^\prime$.  Thus, the 
$\text{dim}(\mathcal{M}(c_2, c_2^\prime)) = \text{dim}(\mathcal{M}(H, c_2^\prime)) + \text{dim}(\text{Quot}(E; d))$. Now the 
dimension of $\text{Quot}(E; d)$ is $3(c_2-c_2^\prime)$. 
Therefore, 
\[
 \text{dim}(\mathcal{M}(c_2, c_2^\prime)) = \text{dim}(\mathcal{M}(H, c_2^\prime)) + 3(c_2-c_2^\prime).
\]
\begin{proposition}\label{p3}
 If $c_2 \ge 27 $ and $c_2 -c_2^\prime \ge 2$, then  the codimension of 
 $\mathcal{M}(c_2, c_2^\prime)$  in $\overline{\mathcal{M}}(H, c_2)$is at least $2$. 
  \end{proposition}
\begin{proof}
{\bf Case I}: $20 \le c_2^\prime \le c_2 -2$.
  Since the moduli space is good for $c_2 \ge 20$, the dimension of $\mathcal{M}(H, c_2^\prime) = 4c_2^\prime -39$.
 Thus $\text{dim}(\mathcal{M}(c_2, c_2^\prime)) = 4c_2^\prime -39 +3(c_2-c_2^\prime)$.  Hence, the codimesion of 
 $(\mathcal{M}(c_2, c_2^\prime))$  is 
 $ 4c_2 -39 -4c_2^\prime +39 -3(c_2 -c_2^\prime) = c_2 -c_2^\prime \ge 2$.
 
 {\bf Case II}:   $c_2^\prime \le 19$.\\
 Adding $3(c_2-{c_2}^{'})=(81 -3c_2^\prime)$ to the corresponding entry of the third column in the tables I  in \cite{DS}, we get the following table:
 \newpage
\begin{table}
\begin{center}
\begin{tabular}{ |c|c|c|c|} 
 \hline
 ${c_2}^{'}$ & e.d & $dim(\mathcal{M}) \le$ & $dim(\mathcal{M}(27, c_2^{'}) \le$  \\ 
5 & -19 &  2 & 68 \\ 
 6 & -15 &  3 & 66 \\ 
 7 & -11 & -1 & 59  \\
 8 & -7  &  7 & 64   \\
 9 & -3 & 10 & 64  \\
 10 & 1 & 11 & 62  \\
  11 & 5 & 13 & 61  \\
  12 & 9 & 19 & 64 \\
  13 & 13& 24 & 66 \\
  14 & 17 & 26 & 65 \\
  15 & 21 & 28 & 29 \\
  16 & 25 & 30 & 63 \\
  17 & 29 & 34 & 64 \\
  18 & 33 & 38 & 65 \\
  19 & 37 & 42 & 66  \\
  
 \hline
\end{tabular}

\caption{upper bounds of the dimensions of $\mathcal{M}(27, {c_2}^{'})$}
\label{table:1}
\end{center}

\end{table}

Thus $(\mathcal{M}(c_2, c_2^\prime))$ has  codimesion 
 $ \ge  2$,  which  concludes the Theorem.
\end{proof}

\begin{theorem}\label{THM2}
 $\mathcal{M}(H, c_2)$ is irreducible for $c_2 \ge 27$. 
 \end{theorem}
 \begin{proof}
 Note that, since $\mathcal{M}(H, c_2)$ is good and each boundary strata has  dimension smaller than the expected dimension, $\overline{\mathcal{M}}(H, c_2)$ is also good.
 Since $\mathcal{M}(H, c_2)$ is good, it is local complete intersection \cite[Lemma 2.1]{SIM}. 
  If possible, let $X$ and $Y$ be two distinct irreducible components of $\mathcal{M}(H, c_2)$. Then the  intersection 
  $X\cap Y = \varnothing$,
    otherwise, it would be of codimension one and contained in the singular locus \cite[Theorem 4.1, Corollary 4.2]{SIM}.
    But if $c_2 \ge 21$ then by Proposition \ref{p2} the singular locus is of codimension at least $2$, a contradiction.
  
  Since $\mathcal{M}(H, c_2)$ is good, there is a bijection between the irreducible components of $\mathcal{M}(H, c_2)$ and 
  $\overline{\mathcal{M}}(H, c_2)$. But $\overline{\mathcal{M}}(H, c_2)$ is connected. Therefore, there exists two 
  irreducible components $\overline{X}$ and $\overline{Y}$, where $X$ and $Y$ 
  are two distinct irreducible components of $\mathcal{M}(H, c_2)$ 
  such that  their intersection $Z:= \overline{X} \cap \overline{Y} \ne \varnothing$.
  Again since $\overline{\mathcal{M}}(H, c_2)$ is good, $Z$ is of codimension one and contained in the singular locus and also
  contained in $\partial{\mathcal{M}(H, c_2)}$ which is also of codimension $1$ in 
  $\overline{\mathcal{M}}(H, c_2)$ \cite{OG}.

  On the other hand,  by Proposition \ref{p3}, the boundary $\partial{\mathcal{M}(H, c_2)}$ is a union of the stratum 
  $\mathcal{M}(c_2, c_2-1)$ of co-dimension
  $1$, plus other strata of strictly smaller dimension.
  Thus $Z$ intersects $\mathcal{M}(c_2, c_2-1)$ in an open set, which is a contradiction, as the general point of 
  $\mathcal{M}(c_2, c_2-1)$ is a smooth point of $\overline{\mathcal{M}}(H, c_2)$ [Remark \ref{R2}].
  Hence, the theorem follows.
  
 \end{proof}
 Then we have the following Corollary:
\begin{corollary}
 Under the hypothesis of Theorem  \ref{THM2}, $\overline{\mathcal{M}}(H, c_2)$ is irreducible. 
 \begin{proof}
 Note that, as in proposition \ref{p3} $\text{dim}(\mathcal{M}(c_2, c_2 -1)= 4c_2 -4 -39 +3= 4c_2 -40$ for $c_2 \ge 21$ and by proposition \ref{p3}, other strata of the boundary have strictly smaller dimension.
 Thus,  no component of $\overline{\mathcal{M}}(H, c_2)$ can be entirely contained in the boundary. In other words, the there is a bijection between the irreducible components of $\overline{\mathcal{M}}(H, c_2)$ and that of $\mathcal{M}(H, c_2)$. Since by theorem \ref{THM2}, the later one is irreducible, $\overline{\mathcal{M}}(H, c_2)$ is irreducible.
 \end{proof}
\end{corollary}


\section*{Acknowledgement}
I wish to express my thanks to Dr. Arijit Dey for introducing me to this problem and for many important discussions. I also would like to thank Prof. Carlos Tschudi Simpson  for many important discussions with him  and  for many important suggestions and comments.

\end{document}